\documentclass[review]{elsarticle}

\usepackage{lineno,hyperref}
\usepackage{amsmath, amsthm, amssymb}

\newtheorem{problem}{Problem}
\usepackage{graphicx}
\usepackage{setspace}
\usepackage{enumitem}
\usepackage{booktabs}
\usepackage{listings}
\usepackage{color}

\definecolor{codegreen}{rgb}{0,0.6,0}
\definecolor{codegray}{rgb}{0.5,0.5,0.5}
\definecolor{codepurple}{rgb}{0.58,0,0.82}
\definecolor{backcolour}{rgb}{0.95,0.95,0.92}

\lstdefinestyle{mystyle}{
	backgroundcolor=\color{backcolour},   
	commentstyle=\color{codegreen},
	keywordstyle=\color{magenta},
	numberstyle=\tiny\color{codegray},
	stringstyle=\color{codepurple},
	basicstyle=\footnotesize,
	breakatwhitespace=false,         
	breaklines=true,                 
	captionpos=b,                    
	keepspaces=true,                 
	numbers=left,                    
	numbersep=5pt,                  
	showspaces=false,                
	showstringspaces=false,
	showtabs=false,                  
	tabsize=2
}
\newcommand{\bb}{\mathbb}

\newcommand{\inv}{^{-1}}

\newcommand{\Z}{\bb Z}

\renewcommand{\S}{\mathcal{S}}
\newcommand{\T}{\mathcal{T}}
\renewcommand{\L}{\mathcal {L}}

\newcommand{\reduce}{\mathrm{replace}}
\newcommand{\replace}{\mathrm{replace}}
\renewcommand{\a}{\alpha}
\renewcommand{\b}{\beta}
\newcommand{\lcm}{\mathrm{lcm}}

\newcommand{\ab}{\alpha^{[\beta/g]}\cdot(-\beta)^{[\alpha/g]}}
\newtheorem{theorem}{Theorem}[section]
\newtheorem{lemma}[theorem]{Lemma}
\newtheorem{proposition}[theorem]{Proposition}

\newtheorem{fact}[theorem]{Fact}
\begin{document}
	\begin{frontmatter}
		
		\title{An analogue of the Erd\H{o}s-Ginzburg-Ziv Theorem over $\Z$}

		\author[aaron]{Aaron Berger}
		
		\address{Yale University, 10 Hillhouse Ave, New Haven, CT 06511}
		\fntext[aaron]{aaron.berger@yale.edu}
		
		
			\begin{abstract}
				Let $\S$ be a multiset of integers. We say $\S$ is a \emph{zero-sum sequence} if the sum of its elements is 0. We study zero-sum sequences whose elements lie in the interval $[-k,k]$ such that no subsequence of length $t$ is also zero-sum. Given these restrictions, Augspurger, Minter, Shoukry, Sissokho,  Voss show that there are arbitrarily long $t$-avoiding, $k$-bounded zero-sum sequences unless $t$ is divisible by $\lcm(2,3,4,\dots,\max(2,2k-1))$. We confirm a conjecture of these authors that for $k$ and $t$ such that this divisibility condition holds, every zero-sum sequence of length at least $t+k^2-k$ contains a zero-sum subsequence of length $t$, and that this is the minimal length for which this property holds.
			\end{abstract}

		
	\end{frontmatter}

	\section{Introduction}
	A famous result of Erd\H{o}s, Ginzburg, and Ziv \cite{EGZ61} states that in any sequence of $2n-1$ elements of $\Z/n\Z$, there is some subsequence that has length $n$ and sums to 0. This result initiated an area of research that focuses on zero-sum sequences over groups. In this paper, we prove a conjecture of Augspurger, Minter, Shoukry, Sissokho,  Voss in \cite{S16} that replicates this theorem for zero-sum sequences over the integers.
	
	We begin with some notation 
	(following \cite{grynkiewicz2013textbook,geroldinger2006notation}). Let $G$ be an abelian group and $G_0 \subseteq G$ a subset. Let $\S$ be a sequence of elements from $G_0$. The sum of the elements of $\S$ is denoted $\sigma(\S)$, and we say $\S$ is a {zero-sum sequence} if $\sigma(S) = 0$. Let $v_a(S)$ be the multiplicity of the term $a$ in $S$, and $|\S| = \sum v_a(\S)$ the length of $S$. If $\T$ is a sequence such that $v_a(\T) \leq v_a(\S)$ for all $a$, we say $\T$ is a {subsequence} of $\S$ and write $\T\mid \S$. When $G = \Z$, we say $\S$ is {$k$-bounded} if $G_0 \subseteq [-k,k]$, or equivalently if $v_a(\S) = 0$ whenever $|a| > k$.  A zero-sum sequence $\S$ is $t$-containing when there exists a $\T\mid \S$ such that $\sigma(\T) = 0$ and $|\T| = t$. A sequence that is not $t$-containing is $t$-avoiding. Finally, we denote $\alpha^{[n]}$ the sequence consisting of $n$ copies of $\alpha$, and use product notation $\T\cdot\S$ to denote the concatenation of two sequences $\T$ and $\S$. Abusing notation, we define $S \cdot a^{-n}$ to be the deletion of $n$ copies of $a$ from $S$. As a demonstrative example of the notation introduced so far, the sequence $10^{[9]}\cdot(-9)^{[10]}$ is a 10-bounded zero-sum sequence of length 19 and is $t$-avoiding for all $t \notin \{0,19\}$.

	This study of $k$-bounded, $t$-avoiding zero-sum sequences is part of the larger study of generalized Erd\H{o}s-Ginzburg-Ziv constants, defined for some $\mathcal L \subseteq \bb N$ and $G_0$ subset of an abelian group $G$. The constant $s_{\mathcal L}(G_0)$ is defined to be the minimal $\ell$ such that any sequence $\S$ of length $\ell$ whose elements come from $G_0$ contains a zero-sum subsequence $\T \mid \S$ such that $|\T| \in \mathcal L$. When $G_0 = G$ and $\mathcal L = \{\exp(G)\}$, where $\exp(G)$ is the exponent of $G$, this constant is written $s(G)$ and is called the Erdos-Ginzburg-Ziv constant. When $\mathcal L = \bb N$ this is denoted by $D(G_0)$, the Davenport constant. 
	
	When working over the integers, the problem of finding $s_\L(G_0)$ is trivial: If $G_0$ contains a nonzero element, then this constant is  $\infty$. As such, recent work in this area (see \cite{S16,sahs2012zero,sissokho2014note}) has studied a modified Erd\H{o}s-Ginzburg-Ziv constant $s'_\L(G_0)$, defined as the minimum length $\ell$ such that any \emph{zero-sum sequence} of length at least $\ell$ whose elements are in $G_0$ contains a zero-sum subsequence $\T$ such that $|\T| \in \L$.
	For example, Lambert \cite{JL87} shows that $s'_{\bb N}([-k,k]) = \max(2,2k-1)$. 
	The authors of \cite{S16} analyze the constant $s'_\L([-k,k])$ when $\L$ is a single integer $t$ and prove the following result.
	\clearpage
	\begin{theorem}[Augspurger, Minter, Shoukry, Sissokho,  Voss \cite{S16}]\label{theirTheorem}~
		\begin{itemize}
			\item $s'_t([-k,k]) < \infty$ if and only if $\lcm(2,\ldots,\max(2,2k-1)) \mid t$.
			\item If $s'_t([-k,k]) < \infty$, then $t+k^2-k \le s'_t([-k,k]) \le t+4k^2-10k+6$.
		\end{itemize}
	\end{theorem}
	With this result, the last remaining work regarding these constants is to close the gap between the bounds on $s'_t([-k,k])$. The authors of \cite{S16} also show that $s'_t([-k,k]) = t+k^2-k$ when $k \in \{1,2,3\}$, and conjecture that this is true for all $k$. In this paper, we confirm that conjecture, allowing us to state the following theorem.
	\begin{theorem}\label{MyTheorem}
		For all $k \in \bb N$,
		$$
		s'_t([-k,k]) = \begin{cases}
		t+k^2-k & ~~\lcm(2,\ldots,\max(2,2k-1)) \mid t~ \\
		\infty & ~~\text{otherwise.}
		\end{cases}
		$$
	\end{theorem}
	\subsection{Organization of Proof}
	We will show that for $k \ge 4$ and $t$ satisfying the divisibility condition of the theorem, there exists no $k$-bounded, $t$-avoiding sequence $\S$ of length at least $t+k^2-k$. The restriction that $\S$ is $t$-avoiding forces $\S$ to be $(n := |\S|-t)$-avoiding, and $n$ is very small compared to $|\S|$. We can find an $\alpha$ and $-\beta$ that occur very frequently in $\S$. Then, in a prescribed manner, we add many copies of $\a$ and $-\b$ to $\S$, as well as replace other elements of $\S$ with $\a$ and $-\b$, in a procedure that preserves the $n$-avoiding property, though not necessarily the length or the $t$-avoiding property. After this replacement is complete, we show a number of restrictions on the sequence resulting from these operations, and from this we obtain enough information about the sequence to deduce a contradiction of our original assumptions.
	
	\section{Proof of Theorem 1.2}\label{Section 2}
	We begin with an observation that zero-sum sequences from bounded sets must have many positive and negative elements:
	\begin{lemma}\label{ratios}
		Let $\S$ be a $k$-bounded zero-sum sequence. Let $\S^+$ be the subsequence of strictly positive elements of $\S$, let $\S^-$ be the strictly negative elements, and $S^0$ be the zeros. Then
		\begin{enumerate}
			\item[i.] $|\S^+|,|\S^-| \le \frac{k}{k+1}(|\S| - |\S^0|) \leq \frac{k}{k+1}|\S|$.
			\item[ii.] $|\S^+|, |\S^-| \ge \frac{1}{k+1}(|\S| - |\S^0|)$.
		\end{enumerate}
	\end{lemma}
	\begin{proof}
		We can bound the positive elements of $\S$ by $k$ and the negative elements by $-1$. This tells us $k|\S^+| \geq \sigma(\S^+)$ and $|\S^-| \leq -\sigma(\S^-) = \sigma(S^+) \leq k|\S^+|$. We use the second inequality to solve for $|\S^-|$ in the equation $|\S^+|+|\S^-| = |\S|-|\S^0|$ to get $\frac{k+1}{k}|\S^-| \leq |\S|-|\S^0|$. This shows (i) for the negative case, and by symmetry it holds in the positive case as well. Part (ii) follows directly from (i) by taking complements.
	\end{proof}
	
	We now move into the main body of the proof. Fix $k \ge 4$, $t$ satisfying $\lcm(2,\ldots,\max(2,2k-1)) \mid t$, and $n \in [k^2-k,4k^2-10k+5]$. Assume for the sake of contradiction that there exists a $t$-avoiding sequence $\S$ with length $|\S| = t+n$. Since $\S$ is a zero-sum sequence, a subsequence is zero-sum if and only if its complement is zero-sum. Thus, since $\S$ is $t$-avoiding it is also $n$-avoiding.

	\begin{lemma}\label{ManyTimes}
		There exist $\alpha > 0, -\beta < 0$ such that $v_\alpha(S),v_{-\beta}(S) \ge \frac{k}{k+1}n$.
	\end{lemma}
	\begin{proof}
		The proof is essentially a pigeonhole argument using the fact that $|\S| \gg n$. We begin by observing $v_0(\S) < n$; otherwise $\S$ would be $n$-containing. Consequently, $\S$ contains at least $t$ nonzero elements.
		Of these, at least $\frac{1}{k+1}t$ must be positive by Lemma \ref{ratios}. By symmetry the same inequality holds for $S^-$.
		
		By pigeonhole principle, at least $\frac{1}{k}$ of these positive elements must be a single element $\alpha$. Hence, recalling the assumption $n \leq 4k^2-10k+5 \leq 4k^2$, it suffices to show $\frac{1}{k(k+1)}t \ge \frac{k}{k+1}4k^2$. By assumption $t \ge \lcm(2,\ldots,2k-1)$, so we reduce to showing $\lcm(2,\ldots,2k-1) \ge 4k^4$. We can check this holds for $k = 5,6,7,8$ by hand. For $k \ge 8$, we use the estimate provided in Equation 5 of \cite{nair1982chebyshev} to obtain  $\lcm(2,\ldots,2k-1) \ge 4^{k-1}$. Inductively, $4^{k-1} \ge 4k^4$ for $k \ge 8$.
		
		The only remaining case, $k=4$, is handled via a more careful application of the argument above and can be found in Lemma \ref{k=4 case} of the Computations section.
	\end{proof}
	It will turn out that $\frac{k}{k+1}n$ occurrences is precisely enough to allow us to pass the threshold of being ubiquitous in some sense--once we have this many copies of $\alpha$ and $-\beta$, it is equivalent to having an arbitrary number, as we will see in the following lemma. Let $g = \gcd(\a,\b)$  and define $X := \ab$ (note $X$ is zero-sum). 
	\begin{lemma}\label{adding copies still n avoiding}
		For any $N \in \bb N$, the sequence $\S' := \S \cdot X^{[N]}$ is $n$-avoiding.
	\end{lemma} 
	\begin{proof}
		Assume for the sake of contradiction that $\S'$ is $n$-containing. Let $\T \mid \S'$ such that $\sigma(\T) = 0$, $|\T| = n$. From Lemma \ref{ratios} we know $|\T^+|,|\T^-| \le \frac{k}{k+1}n$. Then
		\begin{enumerate}
			\item For all $x \notin\{\a,-\b\}$, $v_x(\T) \le v_x(\S') \le v_x(\S)$
			\item For $x \in \{\a,-\b\}$, $v_x(\T) \le |T^\pm| \le \frac{k}{k+1}n \le v_x(\S)$,
		\end{enumerate}
		where the final inequality follows from Lemma \ref{ManyTimes}. Hence $\T | \S$ and so $\S$ is $n$-containing, which is our desired contradiction. 
	\end{proof}
	
	For any zero-sum sequence $S$, if there existsa $T \mid S$ satisfying $\sigma(T) = 0$ $|T| = j|X|$ for some $j \in \bb N$, and $v_x(T) > 0$ for some $x \notin \{\a,-\b\}$ define the following operation:
	$$
	\replace_{\a\b}(S) := 
	S\cdot T\inv \cdot X^{[j]}.
	$$
	When such a $T$ does not exist, let $\replace_{\a\b}(S) := 
	S$.
	
	We will fix $N = 4k^3$ 
	and let $\S' = \S \cdot X^{[N]}$ (in practice this should be thought of as fixing $N \gg_k 1$).
	Applying the definition of $\replace$, see that $|\reduce_{\a\b}(\S')| = |\S'|$ and $\sigma(\reduce_{\a\b}(\S')) = \sigma(\S') = 0$. Moreover $\reduce_{\a\b}(\S') \mid \S' \cdot X^{[j]} \mid \S \cdot X^{[N+j]}$ for some $N,j \in \bb N$ and so it is $n$-avoiding by Lemma \ref{adding copies still n avoiding}. 
	
	Inductively, the $r$-fold repeated application $\reduce^r_{\a\b}(\S')$ is an $n$-avoiding zero-sum sequence of the same legnth as $\S'$. Note that each nontrivial application of $\reduce_{\a\b}$ must strictly increase the value of $v_\a(\S')+v_{-\b}(\S')$, a quantity that is bounded by $|\S'|$. Hence there may only be a finite number of nontrivial applications of $\reduce_{\a\b}$ to $\S'$ before it becomes fixed. We will denote this fixed sequence $\mathcal R$.	

Recall the Davenport constant of modulo $N$:
\begin{fact}[Davenport constant of $\Z/N\Z$]\label{Fact}
	In any set of $N$ integers, some nonempty subset sums to a multiple of $N$.
\end{fact}
This allows us to make the following observation.
\begin{proposition}\label{nice construction}
	The sequence $\mathcal R$ contains strictly less than $\a+\b$ elements not equal to $\a$ or $-\b$.
\end{proposition}
\begin{proof}
	Assume for the sake of contradiction $\mathcal R$ contains $\a+\b$ elements not equal to $\a$ or $-\b$; partition them arbitrarily into sequences of length $g$. By Fact \ref{Fact}, each of these sequences contains a nonempty subsequence with sum 0 mod $g$. For each of these $(\a+\b)/g$ subsequences, construct a zero-sum subsequence of $\S$ by taking the subsequence, adding copies of $\a$ until the sum is a non-negative multiple of $\b$, and adding copies of $-\b$ until the sum is 0. 
	We can perform these constructions disjointly; a quick estimate tells us that this process uses at most $kg \leq k^2$ copies of $\alpha$ to achieve positivity, and another $\beta/g \leq k$ to become a multiple of $\beta$. Then another $\alpha/g \leq k$ copies of $-\beta$ to reach 0. Thus for all $\ab < 2k$ sequences, we require at most $(k^2+2k)2k \leq 4k^3$ copies of either $\a$ or $\b$. But by construction, $v_\a(\mathcal \S'), v_{-\b}(\mathcal \S') \geq N = 4k^3$. Since we constructed $(\a+\b)/g$ of these subsequences, some nonempty subset must have total length equivalent to 0 mod $(\a+\b)/g$ (again by Fact \ref{Fact}). The concatenation of this subset of sequences then forms a zero-sum subsequence of $\S$ with length a multiple of $(\a+\b)/g$, which by construction contains elements not equal to $\a$ or $-\b$. So $\replace_{\a\b}$ will act nontrivially on this subsequence, which is our desired contradiction.
\end{proof}
	
	At this point we make use of our original divisibility assumption on $t$. Recall that $|\mathcal R| - n = |\S'|-n = |\S|-n+N|X| = t+N|X|$ for some $N$. Moreover, $|X| \le \a+\b \le 2k-1$ for $\a \neq \b \le k$, and $|X| = 2$ otherwise. Since $\lcm(2,\ldots,\max(2,2k-1)) \mid t$, we are guaranteed that $|X| \mid t$ and so $|X| \mid (|\mathcal R|-n)$.
	Let $\ell$ be the maximum integer such that $X^{[\ell]} \mid \mathcal R$, and define $\mathcal R' = \mathcal R\cdot X^{-\ell}$. Since $\mathcal R$ is $(|\mathcal R|-n)$-avoiding, we must have $\ell < \frac{|\mathcal R|-n}{|X|}$. Correspondingly, $|\mathcal R'| \ge n+|X| = n+ \frac{\a+\b}{g}$. Additionally, by definition of $\mathcal R'$, either $v_\a(\mathcal R') < \frac{\beta}{g}$ or $v_{-\b}(\mathcal R') < \frac{\a}{g}$. Without loss of generality, assume the first condition holds. 

	 At this point, assume for the sake of contradiction that $\beta \ge 2$; we will show $\mathcal R$ contains at least $\a+\b$ elements not equal to $\a$ or $\b$, contradicting Proposition \ref{nice construction}.
	
	If we have $v_{-\b}(\mathcal R') < \frac{\a}{g}$, then $|\mathcal R'| - v_\a(\mathcal R') - v_{-\b}(\mathcal R') \ge n \ge 2k \ge \a+\b$. Otherwise, let $\mathcal R'' = \mathcal R'\cdot \a^{-v_\a(\mathcal R')} \cdot (-\b)^{-\a/g}$. So $\mathcal R''$ is formed by deleting the rest of the copies of $\a$ from $\mathcal R$ and then deleting a sufficient number of copies of $-\b$ so the sum is positive:
	$$\sigma(\mathcal R'') = \sigma(\mathcal R') - \a v_\a(\mathcal R') + \b\frac{\a}{g} > 0 - \a\frac\b g+\b\frac \a g = 0.$$ 
	
	Since $|\mathcal R'| \ge n+ \frac{\a+\b}{g}$ we have $|\mathcal R''| > n$.
	If $v_{-\b}(\mathcal R'') > \frac{k}{k+\b}|\mathcal R''|$, we would have
	
	$$\sigma(\mathcal R'') \leq \sigma((\mathcal R'')^+)-\b\cdot v_{-\b}(\mathcal R'')\leq k\left(\frac{\b}{k+\b}\right)|\mathcal R''|- \beta\left(\frac{k}{k+\b}\right)|\mathcal R''| \leq 0,$$ 
	which is a contradiction.
	
	We conclude $v_{-\b}(\mathcal R'') \leq \frac{k}{k+\b}|\mathcal R''|$, 
	and so $$|\mathcal R''| - v_\a(\mathcal R'') - v_{-\b}(\mathcal R'') = |\mathcal R''| - v_{-\b}(\mathcal R'') \geq \frac{\beta}{k+\beta}(k^2-k).$$ Noting $\a \le k$, it is sufficient to show $\frac{\beta}{k+\beta}(k^2-k) > k+\beta$ to contradict Lemma \ref{nice construction}. Taking derivatives shows this function is increasing in $k$ and $\beta$ for $k \ge 7$, $\beta \ge 2$, and evaluating shows the inequality is satisfied in this case. Hence it remains to show the lemma for $4 \le k \le 6$, $2 \le \beta \le k$. Say such a counterexample to the proposition exists. Recall $|\mathcal R'| \ge n+\frac{\a+\b}{g} \ge k^2-k+\frac{\a+\b}{g}$. Consider the subsequence of $\mathcal R'$ formed by its largest $k^2-k+(\a+\b)/g$ elements. This subsequence must have non-negative sum, contain at most $\b/g-1$ copies of $\a$, and at most $\a+\b-1$ numbers not equal to $\a$ or $-\b$. An exhaustive search shows this is simply not possible (Lemma \ref{computation}).
	
	Hence, by contradiction we conclude $\beta = 1$.
		

	\subsection{Conclusion of Proof}
	\begin{proof}[Proof of Theorem \ref{MyTheorem}]
	By assumption, $\mathcal R'$ contains at most $\b/g-1 = 1/1-1 = 0$ copies of $\a$. Consequently, the subsequence $(\mathcal R')^{\ge 0}:= (\mathcal R')^0 \cdot (\mathcal R')^{+}$ contains no copies of $\a$ or $-\b$ and so recalling Lemma \ref{nice construction}, we have $|(\mathcal R')^{\ge 0}| \leq \a+\b-1 = \a$. Recall also that $|\mathcal R'| \ge n+\frac{\a+\b}{g} = n+a+1$. 
	Finally, since $\mathcal R'$ is zero-sum, $\max\{x \in \mathcal R'\}\cdot|(\mathcal R')^+| \ge \sigma((\mathcal R')^+) = -\sigma((\mathcal R')^-) \ge |(\mathcal R')^-|$, whence
	$$
	|\mathcal R'| = |(\mathcal R')^{\ge 0}| +  |(\mathcal R')^-| \le \a +  \alpha\max\{x \in \mathcal R'\}.
	$$
	Since $\alpha \notin \mathcal R'$, we cannot have $\alpha = \max\{x \in \mathcal R'\} = k$, so this product is bounded above by $k(k-1)$. This gives us 
	$$
	|\mathcal R'| \le \alpha + k(k-1) \le \alpha+n < |\mathcal R'|,
	$$
	which is our desired contradiction.
	\end{proof}
	\section{Extremal Cases and Further Work}
	Now that we know that the longest possible $t$-avoiding sequences are length $t+k^2-k-1$, it is interesting to ask what such sequences may look like. The authors of \cite{S16} provide two such examples of such sequences whose elements come from the set $\{-1,k-1,k\}$ (or equivalently, its negation). In fact, any sequence of the appropriate length with elements from this set is $(k^2-k-1)$-avoiding via a quick application of the Frobenius Coin Problem. We now provide a rough proof to demonstrate that these are the only extremal sequences.
	
	Following the reasoning of the proof but with $n = k^2-k-1$, we find that any sequence in this extremal case has $-\beta = -1$, $\alpha \in \{k-1,k\}$, and has at least $k-1$ copies of the element in $\{k-1,k\}$ that is not $\alpha$. Assume for the sake of contradiction that $\S$ contains another element $c \notin \{-1,k-1,k\}$. We find an $\ell$ such that $\S$ contains $\ell$ elements summing to $n-\ell$. Adding $n-\ell$ copies of $-1$ gives a zero-sum subsequence of length $n$.
	\begin{itemize}
		\item If $c \in [-k,-2]$, we find $k$ elements summing to $n-k$. One of these will be $c$, so the remaining $k-1$ elements must sum to $n-k-c$, which lies somewhere in the range $[k^2-2k+1,k^2-k-1]$. Using $k-1$ elements from $\{k-1,k\}$, we may construct any sum in the range $[(k-1)(k-1),k(k-1)]$ which covers the entire range of possibilities specified above.
		\item If $c \in [0,k-2]$, we find $k-1$ elements summing to $n-(k-1)$. One of these will be $c$, so the remaining $k-2$ elements must sum to $n-(k-1)-c$, which lies somewhere in the range $[k^2-3k+2,k^2-2k]$. Using $k-2$ elements from $\{k-1,k\}$, we may construct any sum in the range $[(k-1)(k-2),k(k-2)]$ which covers the entire range of possibilities specified above.
	\end{itemize}
	Hence the only $t$-avoiding sequences of length $t+k^2-k-1$ are those composed of elements from $\{-1,k-1,k\}$ or $\{1,-k+1,-k\}$.

	The next step in this discussion would be to take the modified zero-sum constants and bring them back to the context of finite abelian groups. This leads naturally to the problem:
	\begin{problem}\label{prob1}
		Let $G$ be a finite group and $t \in \Z, t > 0$. Compute $s'_{t}(G)$.
	\end{problem}
	The most interesting, and tractable, cases of this problem will occur when $s(G)$ is already known \cite{EGZ61, reiher2007kemnitz}:
	\begin{problem}\label{prob2}
		Solve Problem \ref{prob1} for the cases $G = \Z/n\Z$ and $G = (\Z/n\Z)^2$.
	\end{problem}
\section{Computations}
	\begin{lemma}\label{k=4 case}
		Let $\S$ be defined as in Section \ref{Section 2}
		with $k = 4$. Then there exist $\alpha > 0, -\beta < 0$ such that $v_\alpha(\S),v_{-\beta}(\S) \ge \frac{k}{k+1}n$.
	\end{lemma}
	\begin{proof}
		We show such a positive $\alpha$ exists, and by symmetry so must a $-\beta$.
		\begin{enumerate}
			\item If $\sigma(\S^+) \le 3.5|S^+|$, then using $|\S^-| \leq -\sigma(\S^-) = \sigma(\S^+)$ and $|\S^+|+|\S^-|\ge t$ we obtain $|\S^+| + 3.5|\S^+| \ge t$, so $|\S^+| \geq \frac{1}{4.5}t$. Of these, at least $\frac{1}{4}$ must be a single $\alpha$ by pigeonhole. Hence it suffices to check $\frac{t}{18} \geq \frac{4}{5}n$. 
			\item  If $\sigma(\S^+) > 3.5|\S^+|$, then first apply Lemma \ref{ratios}. Recallyng that by assumption, $\S$ contains at least $t$ nonzero elements, we conclude that $|\S^+| \ge \frac{1}{k+1}t = \frac{1}{5}t$. Partitioning $\S^+$ into the elements equal to 4 and those less than 4, we obtain $3.5|\S^+| < \sigma(\S^+) \leq 4v_4(\S^+)+3(|\S^+|-v_4(S^+))$. Simplifying gives $v_4(\S^+) > \frac{1}{2}|\S^+| \ge \frac{1}{10}t$. It then suffices to check $\frac{t}{10} \ge \frac{4}{5}n$.

		\end{enumerate}
		We show $\frac{t}{10} \ge \frac{t}{18} \ge \frac{4}{5}n$ by using $t \ge \lcm(2,\ldots,7) = 420$ and $n \le (4\cdot 4^2-10\cdot 4+5) = 29$. The inequality simplifies to $23.33 \ge 23.2$, so the lemma holds.
	\end{proof}
	\begin{lemma}\label{computation}
		For any $4 \le k \le 6$, $1 \leq \alpha \leq k$, and $2 \le \beta \le k$, there is no sequence of length $k^2-k+(\a+\b)/\gcd(\a,\b)$ with non-negative sum that contains at most $\b/\gcd(\a,\b)-1$ copies of $\a$ and at most $\a+\b-1$ numbers not equal to $\a$ or $-\b$.
	\end{lemma}
	In trying to construct such a sequence, it is always optimal to use as many copies of $\a$ as possible and use $\a+\b-1$ copies of the maximum positive number not equal to $\a$. Afterwards, we fill the remaining spots with copies of $-\b$ and check whether the sum is non-negative.

	It is reasonable to perform this check by hand; we provide some code (Python 3.5) that accomplishes the task.
	\footnote{This implementation deliberately sacrifices computational and character-count efficiency in pursuit of readability.}

	\begin{lstlisting}[language=python]
from fractions import gcd

for k in range(4,7):
  for b in range (2,k+1):
    for a in range(1,k+1):
      g = gcd(a,b)
# maxpos is the largest positive number not equal to a.      
      if a == k:
        maxpos = k-1
      else:
        maxpos = k 
      aCopies = b//g-1
      maxposCopies = a+b-1
      bCopies = (k*k-k+(a+b)//g) - aCopies - maxposCopies
      if a*aCopies + maxpos*maxposCopies - b*bCopies >= 0:
        print ('Found counterexample')
\end{lstlisting}
\section{Acknowledgements}
This research was carried out at the Duluth REU under the supervision of Joe Gallian. Duluth REU is supported by the University of Minnesota, Duluth and by grants NSF-1358659 and NSA H98230-16-1-0026. Special thanks to Levent Alpoge, Phil Matchett Wood, Joe Gallian, Papa Sissokho, and Pat Devlin for helpful comments on the paper. Finally, I am extremely grateful to my referees for going above and beyond the call of duty and performing especially careful readings with many helpful comments, corrections, and suggestions to improve this paper.
	\clearpage

	\bibliography{ZSS_Bibliography}
\end{document}